\documentclass[11pt,twoside,a4paper,english]{amsart}
\usepackage{babel}
\usepackage[utf8]{inputenc}
\usepackage{amscd}
\usepackage{pstricks}
\usepackage{amsmath,amsthm,amssymb}
\usepackage{mathrsfs}
\usepackage{graphicx,psfrag,epsfig}
\RequirePackage{pst-all}
\usepackage{mathrsfs}
\usepackage{float}
\usepackage{hyperref}

\def\pg{\leaders\hbox to 5mm{\hfil.\hfil}\hfill}

\def\text#1{\,\hbox{#1}\;}

\reversemarginpar

\newcounter{paraga}

\def\Bbibitem#1#2{\bibitem[#1]{#2}}

\def\frac#1#2{{\textstyle {#1\over#2}}}

\theoremstyle{plain}

\newtheorem{thm}{Theorem}
\newtheorem{lem}{Lemma}[section]

\newtheorem{prop}{Proposition}[section]
\newtheorem{Def}{Definition}[section]

\newtheorem{cor}{Corollary}[section]

\def\norm#1{\left\Vert#1\right\Vert}

\def\M{{\mathscr M}}
\def\G{{\mathscr G}}

\def\jC{{\mathscr C}}

\def\jD{{\mathscr D}}
\def\jH{{\mathscr H}}

\def\O{{\mathscr O}}

\def\U{{\mathscr U}}

\def\jV{{\mathscr V}}

\def\N{{\Bbb N}}
\def\R{{\Bbb R}}

\def\htop{{\rm h_{top}}}

\def\htop{{\rm h_{\rm top}\,}}

\title[Genericity of weak-mixing measures]{Genericity of weak-mixing measures on geometrically finite manifolds}
\author{Kamel Belarif}
\curraddr{LMBA - 6 Avenue Victor Le Gorgeu \newline%
\indent 29238, Brest}%
\email{kamel.belarif@univ-brest.fr}%

\begin{document}

\begin{abstract}
Let $M$ be a manifold with pinched negative sectional curvature. We show that when $M$ is geometrically finite 
and the geodesic flow on $T^1 M$ is topologically mixing then the set of mixing invariant measures is dense in the set
$\M^1(T^1M)$ of invariant probability measures. This implies that the set of weak-mixing measures which are invariant by 
the geodesic flow is a dense $G_{\delta}$ subset of $\M^1(T^1 M)$. We also show how to extend these results
to manifolds with cusps or with constant negative curvature.
\end{abstract}
\maketitle


\section{Introduction}

Let $\widetilde{M}$ be a complete, simply connected manifold with pinched sectional curvature (\textit{i.e} there exists $b > a > 0$ such that $-b^2 \leq \kappa 
\leq -a^2$) and $\Gamma$ a non elementary group of isometries of $\widetilde{M}$. We will denote by $M$ the quotient manifold $\widetilde{M} / \Gamma$
and $\phi_t$ the geodesic flow on the unit tangent bundle $T^1 M$. The interesting behavior of this flow occurs on its non-wandering set $\Omega$.\\
Let us recall a few definitions related to the mixing property of the geodesic flow.

\begin{enumerate}
 \item $\phi_t$ is topologically mixing if for all open sets $\U, \jV \subset \Omega$ there exists $T>0$ such that 
\begin{center}
 $\phi_t(\U) \cap \jV \neq \emptyset$ for all $|t| >T$,
\end{center}
\item Given a finite measure $\mu$ the geodesic flow is mixing with respect to $\mu$ if for all $f \in L^2(T^1 M,\mu)$,
$$ \lim\limits_{t \to \infty} \int_{T^1M} f \circ \phi_t \cdot f d\mu = \left( \int_{T^1 M} f d\mu \right)^2, $$
\item it is weakly-mixing if for all continuous function $f$ with compact support we have
 $$ \lim\limits_{T \to \infty} \frac{1}{T} \int_0^T \left| { \int_{T^1 M} f \circ \phi_t (x)f(x) d\mu(x) - \left( \int_{T^1 M} f d\mu \right) ^2 } \right| dt = 0.$$
\end{enumerate}

Finally, let us recall what is the weak topology: a sequence $\mu_n$ of probability measures converges to a probability measure $\mu$ if for all bounded
continuous functions $f$, 
$$\int_{T^1 M} f d\mu_n \xrightarrow{n \to \infty} \int_{T^1 M} f d\mu.$$\\

In this article, we will show that the weak-mixing property is generic in the set $\M^1(T^1 M)$ of probability measures invariant by the geodesic flow and supported
on $\Omega$. We endow $\M^1(T^1 M)$ with the weak topology.\\
In \cite{Si72}, K. Sigmund studies this question for Anosov flows defined on compact manifolds and shows
that the set of ergodic probability measures is a dense $G_{\delta}$ set (\textit{i.e} a countable intersection of dense sets) in $\M^1 (T^1 M)$.
On non-compact manifolds, the question has been studied by Y. Coudène and B. Schapira in \cite{MR2735038} and \cite{MR3322793}. It appears that ergodicity and zero entropy are typical 
properties for the geodesic flow on negatively curved manifolds.\\
Since the set of mixing measures with respect to the geodesic flow is contained in a meager set (see \cite{MR3322793}) it is natural to consider the set
of weak-mixing measures from the generic point of view.\\

Let us recall that a manifold $M$ is geometrically finite if it is negatively curved, complete and has finitely many ends, all of which
are cusps of finite volume and funnels.\\

Here is our main result.
\begin{thm}\label{thm:princ}
 Let $M$ be a geometrically finite manifold with pinched curvature and $\phi_t$ the geodesic flow defined on its unit tangent bundle $T^1 M$.\\
 If $\phi_t$ is topologically mixing on $\Omega$, then the set of mixing probability measures invariant by the geodesic flow is dense in $\M^1 (T^1 M)$ 
 for the weak topology.
\end{thm}
\begin{cor}\label{cor:gef}
 Let $M$ be a geometrically finite manifold with pinched negative curvature and $\phi_t$ the geodesic flow defined on its unit tangent bundle $T^1M$.\\
 If $\phi_t$ is topologically mixing on the non-wandering set of $T^1M$, then the set of invariant weak-mixing probability measures with
 full support on $\Omega$ is a dense $G_{\delta}$ subset of $\M^1 (T^1 M)$.
\end{cor}
To prove theorem \ref{thm:princ}, we will use the fact that Dirac measures supported on periodic orbits are dense in $\M^1(T^1 M)$. This comes from \cite{MR2735038}
where the result is shown for any metric space $X$ admitting a local product structure and satisfying the closing lemma.\\
The rest of the proof relies on the approximation of a single Dirac measure supported on a periodic orbit $\O(p)$ using a sequence of Gibbs
measures associated to $(\Gamma,F_n)$ where $F_n:T^1 M \to \R$ is a Hölder-continuous potential.\\
The notion of Gibbs measures which is related to the construction of \\ Patterson-Sullivan densities on the boundary at infinity of $\widetilde{M}$ will be 
recalled in $\S. 2$ .\\
In $\S. 4$ we will prove a criterion connecting the divergence of some subgroups of $\Gamma$ with the finiteness of the Gibbs measures which comes from \cite{MR1776078}
for the potential $F=0$.\\
After this, we will construct in $\S. 5$ a sequence of bounded potentials satisfying the desired property. The main step of this paragraph builds on a result of 
\cite{C04} which claims that there exists a bounded potential such that the Gibbs measure is finite.\\ 
Finally, we will prove in $\S. 6$ the convergence of the Gibbs measures using the variational principle which is recalled in $\S. 2$.\\
Now, assume that Theorem \ref{thm:princ} is true. The proof of corollary \ref{cor:gef} is a consequence of \cite{MR3322793} where it is shown that the set of weak-mixing
measures with full support is a $G_{\delta}$ subset of the set of invariant Borel probability measures supported on $\Omega$.\\

In the previous theorem, we restricted ourselves to the case of geometrically finite manifolds but we make the following conjecture: the result is still
true for non geometrically finite manifolds as soon as the geodesic flow is topologically mixing on its non wandering set.\\
The conjecture is supported by the following two results.

\begin{cor}\label{cor:geinf}
 Let $M$ be a connected, complete pinched negatively curved manifold with a cusp then the set of probability measures 
 fully supported on $\Omega$ that are weakly mixing with respect to the geodesic flow is a dense $G_{\delta}$ set of $\M^1 (T^1 M)$.
\end{cor}

\begin{cor}\label{cor:surf}
 Let $S$ be a pinched negatively curved surface or a manifold with constant negative curvature then the set of probability measures
 fully supported on $\Omega$ that are weakly mixing with respect to the geodesic flow is a dense $G_{\delta}$ set of $\M^1 (T^1 M)$.
\end{cor}
Let $M$ be a manifold such that $dim(M)= 2$, $\kappa_M = -1$ or $M$ possesses a cusp and denote by $\phi_t$ the geodesic flow on $T^1M$. We will show in $\S.7$ that
we can find a geometrically finite manifold $\hat{M}$ on which theorem \ref{thm:princ} applies and for which $\phi_t$ is a factor 
of the geodesic flow $\hat{\phi}_t$ on $T^1 \hat{M}$.\\

One way to confirm the conjecture is to find a positive answer to the following question:\\
Let $M$ be a connected, complete manifold with pinched negative curvature. We will suppose that $M$ is not geometrically finite and has no 
cusp. Does there exist a potential $F: T^1 M \to \R$ such that the Gibbs measure associated with $(\Gamma, F)$ is finite?

\section{Preliminaries}
\subsection{Geometry on $T^1 \widetilde{M}$}
We first recall a few notations and results related to the geometry of negatively curved manifolds.\\
Let $\partial_{\infty} \widetilde{M}$ be the boundary at infinity of $\widetilde{M}$, we define the limit set of $\Gamma$ by
$$\Lambda \Gamma = \overline{ \Gamma x} \cap \partial_{\infty} \widetilde{M},$$
where $x$ is any point of $\widetilde{M}$.\\
An element $\xi_p \in \Lambda \Gamma$  is \textbf{parabolic} if there exists a parabolic isometry $\gamma \in \Gamma$ satisfying $\gamma \xi = \xi$.\\
A parabolic point $\xi \in \Lambda \Gamma$ is \textbf{bounded} if $\Lambda \Gamma / \Gamma_{\xi_p}$ is compact where $\Gamma_{\xi_p}$ is the maximal subgroup of $\Gamma$
fixing $\xi_p$. In this case, let $\jH_{\xi}$ be a horoball
centered at $\xi$ then,
$$ \jC_{\xi} = \jH_{\xi} / \Gamma_{\xi_p} $$
is called the cusp associated with $\xi$.\\
We say that $\Lambda_c \Gamma \subset \Lambda \Gamma$ is the conical limit set if for all $\xi \in \Lambda_c \Gamma$ for some 
$\widetilde{x} \in \widetilde{M}$ there exists $\epsilon >0$ and a sequence $(\gamma_n)_{n \in \N}$ such that $(\gamma_n (\widetilde{x}))_{n \in \N}$
converges to $\xi$ and stays at distance at most $\epsilon$ from the geodesic ray $(\widetilde{x} \xi)$.\\
We define the parabolic limit set as follows.

$$\Lambda_p \Gamma = \{ \eta \in \Lambda \Gamma : \exists \gamma \in \Gamma \text{ parabolic }, \gamma \cdot \eta = \eta \}$$

Let us choose an origin $\widetilde{x}_0$ in $\widetilde{M}$ once and for all.
We define the Dirichlet domain of $\Gamma$, centered on $\widetilde{x}_0$ as follows.
$$ \jD = \bigcap\limits_{\gamma \in \Gamma, \gamma \neq Id} \{ \widetilde{x} \in \widetilde{M} : d(\widetilde{x},\widetilde{x}_0) \leq d(\widetilde{x},\gamma \widetilde{x}_0) \}.$$

It is a convex domain having the following properties.
\begin{itemize}
 \item $\bigcup\limits_{\gamma \in \Gamma} \gamma \jD = \widetilde{M},$
 \item for all $\gamma \in \Gamma \backslash \{Id\}, \mathring{\jD} \cap \gamma \mathring{\jD} = \emptyset.$
\end{itemize}

We define the diagonal of $\Lambda \Gamma \times \Lambda \Gamma$ the set of points $(x,y) \in \Lambda \Gamma \times \Lambda \Gamma$ such that $x=y$.
We denote by $\Delta$ this set.\\
For all $\xi, \eta \in \partial_{\infty} \widetilde{M}$, we denote by $(\xi \eta)$ the geodesic joining $\xi$ to $\eta$. We define the lift of the non
wandering set on $T^1 \widetilde{M}$ by
$$\widetilde{\Omega} = \{\widetilde{x} \in T^1 \widetilde{M}: \exists (\xi,\eta)\in \Lambda \Gamma \times \Lambda \Gamma \backslash \Delta, \text{  } \widetilde{x} \in
 (\xi\eta) \}.$$
Let $\pi:T^1\widetilde{M} \to \widetilde{M}$ be the natural projection of the unit tangent bundle to the associated manifold. We denote by $\jC\Lambda\Gamma$ the 
smallest convex set in $\widetilde{M}$ containing $\pi(\widetilde{\Omega})$.

$M$ is geometrically finite if one of the following equivalent conditions is satisfied
\begin{enumerate}
 \item $ \Lambda \Gamma = \Lambda_c \Gamma \cup \Lambda_p \Gamma $ = $\Lambda_c \Gamma \cup \{ \text{bounded parabolic fixed points} \}$,
 \item For some $\epsilon >0$, the $\epsilon-$neighborhood of $\jC \Lambda\Gamma / \Gamma$ has finite volume,
 \item $M$ has finitely many ends, all of which are funnels and cusps of finite volume.
\end{enumerate}

We define a map
$$ C_F: \partial_{\infty} \widetilde{M} \times \widetilde{M}^2 \to \R $$
called the Gibbs cocycle of $(\Gamma,F)$ by
$$ C_{F,\xi}(x,y)= C_F (\xi,x,y) = \lim\limits_{t \to \infty} \int_y^{\xi(t)}\widetilde{F} - \int_x^{\xi(t)} \widetilde{F}$$
where $t \mapsto \xi(t)$ is any geodesic ending at  $\xi$.\\

Here is a technical lemma of \cite{PPS} giving estimates for the Gibbs cocycle.
\begin{prop}\label{prop:major}(\cite{PPS})
 For every $r_0>0$, there exists
$c_1,c_2,c_3,c_4>0$ with $c_2,c_4\leq 1$ such that the following assertions hold.

(1) For all $x,y\in \widetilde{M}$ and $\xi\in\partial_\infty \widetilde{M}$,
$$
|\;C_{F,\,\xi}(x,y)\;|\leq 
\;c_1\,e^{d(x,\,y)} \;+\; d(x,y)\max_{\pi^{-1}(B(x,\,d(x,\,y)))}|\widetilde{F}|\;,
$$
and if furthermore $d(x,y)\leq r_0$, then 
$$
|\;C_{F,\,\xi}(x,y)\;|\leq
c_1\,d(x,y)^{c_2}+ d(x,y)\max_{\pi^{-1}(B(x,\,d(x,\,y)))}|\widetilde{F}|\;.
$$

(2) For every $r\in\mathopen{[}0,r_0\mathclose{]}$, for all $x,y'$ in
$\widetilde{M}$, for every $\xi$ in the shadow $\O_xB(y',r)$ of the ball
$B(y',r)$ seen from $x$, we have
$$
\Big|\;C_{F,\,\xi}(x,y')+\int_x^{y'} \widetilde{F}\;\Big|\leq 
c_3 \;r^{c_4}+2r\max_{\pi^{-1}(B(y',\,r))}|\widetilde{F}|\;. 
$$
\end{prop}
\subsection{Thermodynamic formalism for negatively curved manifolds}
We start by recalling a few facts on Gibbs measures on negatively curved manifolds. The results of this paragraph come from \cite{PPS}.

Let $\widetilde{F}: T^1 \widetilde{M} \to \R$ be a $\Gamma-$invariant Hölder function. We will say that the induced function $F$ on
$T^1 M = T^1 \widetilde{M} / \Gamma$ is a potential.\\

The Poincaré series associated with $(\Gamma,F)$ is defined by
$$ P_{x,\Gamma,F} (s) = \sum\limits_{\gamma \in \Gamma} e^{\int_x^{\gamma x}\widetilde{F}-s}.$$
Its critical exponent is given by
$$ \delta_{\Gamma,F} = \limsup\limits_{n \to \infty} \frac{1}{n} \log (\sum\limits_{\gamma \in \Gamma, n-1 \leq d(x,\gamma x)\leq n} 
e^{\int_x^{\gamma x} \widetilde{F}}).
$$
We say that $(\Gamma,F)$ is of divergence type if\\ $P_{x,\Gamma,F}(\delta_{\Gamma,F})$ diverges.\\
When $F=0$ on $T^1M$, we will denote by $\delta_{\Gamma}$ the critical exponent associated to $(\Gamma,F)$.

\begin{prop} \label{prop:exp}
 Let $F$ be the potential on $T^1 M = (T^1 \widetilde{M}) / \Gamma$ induced by the $\Gamma-$invariant potential
 $\widetilde{F}:T^1 \widetilde{M} \to \R$.
 \begin{enumerate}
  \item The Poincaré series associated with $(\Gamma,F)$ converges if $s >\delta_{\Gamma,F}$ and diverges if $s<\delta_{\Gamma,F}$,
  \item We have the upper bound
  $$ \delta_{\Gamma,F} \leq \delta_{\Gamma} + \sup\limits_{\pi^{-1}(\jC \Lambda \Gamma)} \widetilde{F},$$
  \item For every $c>0$, we have
  $$ \delta_{\Gamma,F} = \limsup\limits_{n \to \infty} \frac{1}{n} \log (\sum\limits_{\gamma \in \Gamma, n-c \leq d(x,\gamma x)\leq n} 
e^{\int_x^{\gamma x} \widetilde{F}}).$$
 \end{enumerate}
\end{prop}

We define a set of measures on $\partial_{\infty} \widetilde{M}$ as the limit points when $s \to \delta_{\Gamma,F}^+$ of
$$ \frac{1}{P_{x,\Gamma,F}(s)} \sum\limits_{\gamma \in \Gamma} e^{\int_x^{\gamma x} \widetilde{F}-s} h(d(x,\gamma x)) D_{\gamma x} = 
\mu^F_{x,s}.$$
where $h: \R_+ \to \R_+^*$ is a well chosen non-decreasing map and $D_{\gamma x}$ is the Dirac measure supported on $\gamma x$.
      
\begin{prop} 
 If $\delta_{\Gamma,F}< \infty$ then
 \begin{enumerate}
  \item $\{ \mu^F_{x,s} \}$ has at least one limit point when $s \to \delta_{\Gamma,F}^+$ with support $\Lambda \Gamma$,
  \item If $\mu^F_x$ is a limit point then it is a Patterson density \textit{i.e}\\
  $\forall \gamma \in \Gamma, x,y \in \widetilde{M}, \xi \in \partial_{\infty} \widetilde{M}$\\
  $$ \gamma_* \mu^F_x = \mu^F_{\gamma x}, $$
  $$ d\mu^F_x (\xi) = e^{-C_{F-\delta_{\Gamma,F}, \xi}(x,y)} d\mu^F_y (\xi). $$
 \end{enumerate}
\end{prop}

Using the Hopf parametrization on $T^1 \widetilde{M}$, each unit tangent vector $v$ can be written as $v= (v_+,v_-,t) \in
\partial_{\infty} \widetilde{M} \times \partial_{\infty} \widetilde{M} \times \R$. We define a measure on 
$T^1 \widetilde{M}$ by
$$ d \widetilde{m}_{\widetilde{F}} (v) = \frac{d \widetilde{\mu}_x^{\widetilde{F} \circ \iota} (v_-) 
d \widetilde{\mu}_x^{\widetilde{F}} (v) dt}{D_{F,x} (v_+,v_-) } $$

where 
$$D_{F,x} (v_+,v_-) = e^{- \frac{1}{2} (C_{F,v_-} (x, \pi (v)) + (C_{F\circ \iota,v_+} (x,\pi(v)))} $$ 
is the potential gap and 
$$ \iota \left\{
      \begin{aligned}
	T^1 \widetilde{M} &\to &T^1 \widetilde{M}\\
	v &\mapsto& -v
	\end{aligned}
\right.$$
is the antipodal map.\\
This measure is called the Gibbs measure associated to $(\Gamma,F)$.\\
It is a measure independent of $x$, invariant under the action of $\Gamma$ and invariant by the geodesic flow. Hence it
defines a measure $m^F$ on $T^1M$ invariant by the geodesic flow.\\

Let $m \in \M^1 (T^1 M)$ be a measure with finite entropy $h_{m}(\phi_t)$. We define the \textbf{metric pressure} of
a potential $F$ with respect to the measure $m$ as the quantity
$$ P_{\Gamma,F}(m) = h_{m}(\phi_t) +\int_{T^1 M} F dm .$$
We say that the supremum
$$ P(\Gamma,F) = \sup\limits_{m \in \M(T^1M)} P_{\Gamma,F}(m) $$
is the \textbf{topological pressure} of the potential $F$. An element realizing this upper bound is called an \textbf{equilibrium state} for 
$(\Gamma,F)$.

\begin{thm}\cite{otal2004, PPS}\label{thm:varp}
 Let $\widetilde{M}$ be a complete, simply connected Riemannian manifold with pinched negative curvature, $\Gamma$ a non-elementary discrete group of isometries of $\widetilde{M}$
 and $\widetilde{F}: T^1 \widetilde{M} \to \R$ a Hölder-continuous $\Gamma$-invariant map with $\delta_{\Gamma,F} < \infty$.
 \begin{enumerate}
  \item We have
  $$ P(\Gamma,F) = \delta_{\Gamma,F}.$$
  \item If there exists a finite Gibbs measure $m_F$ for $(\Gamma,F)$ such that the negative part of $F$ is $m_F -$integrable, then 
  $m^F = \frac{m_F}{\norm{m_F}}$ is the unique equilibrium state for $(\Gamma,F)$. Otherwise, there exists no equilibrium state for 
  $(\Gamma,F)$.
 \end{enumerate}
\end{thm}

\section{Mixing property for the geodesic flow}
The question of the topological mixing of the geodesic flow on a negatively curved manifold is still open in full generality. This question is closely related to
mixing with respect to a Gibbs measure (see \cite{PPS}).\\
We define the length of an element $\gamma \in \Gamma$ by $\ell (\gamma) = \inf\limits_{z \in M} d(z, \gamma z)$.
\begin{thm}
 If $\delta_{\Gamma,F} < \infty$ and $m^F$ is finite then the following propositions are equivalent.
 \begin{enumerate}
  \item The geodesic flow is topologically mixing on $\Omega $,
  \item The geodesic flow is mixing with respect to $m^F$,
  \item $L(\Gamma) =  \{ \ell(\gamma); \gamma \in \Gamma \}$ is not contained in a discrete subgroup of $\R$.
 \end{enumerate}
\end{thm}

Here are some cases where the geodesic flow is known to be topologically mixing \cite{MR1617430},\cite{MR1703039},\cite{MR1779902}.

\begin{lem}\label{lem:mix}
 Let $\Gamma$ be a non elementary group of isometries of a Hadamard manifold $\widetilde{M}$ with pinched negative curvature.
 If $M = \widetilde{M} / \Gamma$ satisfies one of the following properties then the restriction of the geodesic flow to its non-wandering set is topologically mixing.
 \begin{enumerate}
  \item The curvature of $M$ is constant,
  \item $\dim M = 2$,
  \item There exists a parabolic isometry in $\Gamma$,
  \item $\Omega = T^1 M$.
 \end{enumerate}
\end{lem}

To conclude this section, let us recall the Hopf-Tsuji-Sullivan criterion for the ergodicity of the geodesic flow with
respect to the Gibbs measure (see \cite{PPS} for a proof)

\begin{thm}\label{thm:HTS}
 The following assertions are equivalent:
 \begin{enumerate}
  \item $(\Gamma,F)$ is of divergence type,
  \item $\forall x\in \widetilde{M}, \mu_x^{F} (\partial_{\infty} \widetilde{M} \backslash \Lambda_c \Gamma) = 0$,
  \item The dynamical system $(T^1M, (\phi_t)_{t \in \R},m_F)$ is ergodic.
 \end{enumerate}
\end{thm}

As a consequence of this theorem, one can show that if $\delta_{\Gamma,F} < \infty$, the Patterson density $(\mu_x^{F})_{x \in \widetilde{M}}$ 
associated with $(\Gamma,F)$ is non-atomic (see \cite{PPS} Proposition 5.13).

\section{A finiteness criterion}
 
 First, let us give a criterion for the finiteness of the Gibbs measure. This result comes from \cite{MR1776078} for a potential $F=0$. For 
 the general case where $F$ is an Hölder potential, the proof is given in \cite{PPS}.

 \begin{thm}
  Suppose that $\Gamma$ is a geometrically finite group with $(\Gamma,F)$ of divergence type and $\delta_{\Gamma,F}< \infty$.
  The Gibbs measure $m_{F}$ is finite if and only if for every parabolic fixed point $\xi_p$
 \begin{center}
   $\sum\limits_{\gamma \in \Gamma_{\xi_p}} d(x,\gamma x) e^{\int_x^{\gamma x} (\widetilde{F}-\delta_{\Gamma,F})} $ 
 \end{center}
 converges.
\end{thm}

\begin{Def}
 $(\Gamma ,F)$ satisfies the spectral gap property if for all parabolic points $ \xi_p \in \partial_{\infty} \widetilde{M}$,
 $$ \delta_{\Gamma_{\xi_p},F} < \delta_{\Gamma,F}. $$
\end{Def}

Proposition 2 of \cite{MR1776078} gives a criterion for this property for the zero potential. The following proposition is more general and 
applies to all Hölderian potentials.

\begin{prop}
 Let $\widetilde{M}$ be a Hadamard manifold with pinched negative curvature and $\Gamma$ a geometrically finite discrete group acting on
 it. Suppose there exists a bounded Hölderian potential $\widetilde{F}: T^1 \widetilde{M} \to \R$.\\
 If for all parabolic fixed point $\xi_p$ the couple $(\Gamma_{\xi_p},F)$ is of divergence type, then $(\Gamma,F)$ satisfies the spectral gap property.
 Moreover, the Gibbs measure $m^F$ associated to $(\Gamma,F)$ is finite.
\end{prop}

\begin{proof}
To prove the first claim, we follow the ideas of \cite{MR1776078} when $F=0$.\\
Since the action of $\Gamma_{\xi_p}$ on $\partial_{\infty} \widetilde{M}$ has a fundamental domain $\G$ in $\partial_{\infty} \widetilde{M}$, we have
$$
\mu_x^F (\partial_{\infty} \widetilde{M}) = \sum\limits_{g \in \Gamma_{\xi_p}} \mu_x^F (g \G) + \mu_x^F (\Lambda \Gamma_{\xi_p}).
$$
Moreover, since there exists $K\in \R$ such that
$$
\left\{
      \begin{aligned}
	\mu_x^F (g \G) = \int_{\G} e^{-C_{F-\delta_{\Gamma,F}, \xi}(x,gx)} d\mu_{gx} (\xi)\\
	|C_{F,\xi} (x,gx) + \int_x^{gx} \widetilde{F}| \leq K,
      \end{aligned}
\right.
$$
we have
$$ \mu_x^F (g \G) \geq (e^{\int_x^{gx} \widetilde{F} - \delta_{\Gamma,F}}) ( e^{-K} \mu_x^F(\G) ) $$
and
$$ \infty > \mu_x^F (\partial_{\infty} \widetilde{X}) \geq \sum\limits_{g \in \Gamma_{\xi_p}} \mu_x^F (g \G)
   \geq C_0 \cdot P_{x,\Gamma_{\xi_p},F}(\delta_{\Gamma,F}).
$$
So, $\delta_{\Gamma,F} > \delta_{\Gamma_{\xi_p},F}.$\\

For the second claim, since $(\Gamma,F)$ satisfies the spectral gap property for $M = \widetilde{M} / \Gamma$ and
$$ \forall \epsilon >0,  \exists C_{\epsilon}>0, d(x,\gamma x) \geq C_{\epsilon} \Rightarrow 
   e^{\epsilon d(x,\gamma x)} > d(x,\gamma x), $$
we have

$$ \sum\limits_{\gamma \in \Gamma_{\xi_p}} d(x,\gamma x) e^{\int_x^{\gamma x} \widetilde{F} - \delta_{\Gamma, F}} \leq \sum\limits_{\gamma \in \Gamma_{\xi_p}} 
e^{\int_x^{\gamma x} \widetilde{F} - (\delta_{\Gamma, F} - \epsilon)} $$

Choosing $\epsilon$ small enough such that $\frac{\delta_{\Gamma,F} - \delta_{\Gamma_{\xi_p},F}}{2} > \epsilon$,
the series converges and the Gibbs measure is finite.
\end{proof}

\section{Construction of the potentials}
We now construct a $\Gamma-$invariant potential $\widetilde{H}: T^1 \widetilde{M} \to \R$ such that the associated Gibbs measure is finite and which critical
exponent associated to $(\Gamma_{\xi_p},H)$ is of divergence type.\\

Since $M$ is geometrically finite, the set $Par_{\Gamma}$ of parabolic points $\xi_p \in \partial_{\infty} \widetilde{M}$ intersecting the boundary of
the Dirichlet domain is finite.\\
We define for those parabolic points a family of disjoint horoballs $\{\jH_{\xi_p} (u_{0,\xi_p}) \}_{\xi_p}$
on $\widetilde{M}$ passing through a well chosen point $\widetilde{u}_{0,\xi_p}$ of the cusp.

For any $\widetilde{u}$ in $\jH_{\xi_p} (\widetilde{u}_{0,\xi_p})$, we define a height function $\rho: \widetilde{M} \to \R$
by the Buseman cocycle at $\widetilde{u}_{0,\xi_p}$:
$$ \rho(\widetilde{u}) = \beta_{\xi_p}(\widetilde{u},\widetilde{u}_{0,\xi_p}).$$
This cocycle coincides with the Gibbs cocycle for the potential $F= -1$.\\
The curve levels of this function are the horocycles based on $\xi_p$ passing through $\widetilde{u}$.\\

Let $t_n$ be a decreasing sequence of positive numbers converging to $0$.\\
One can construct a sequence $Y_n$ of positive numbers such that

$$
\left\{
      \begin{aligned}
	Y_{n+1} \geq Y_n + t_n - t_{n+1}, \\
	\sum \limits_{p \in \rho^{-1}(]Y_n , Y_{n+1}])} e^{d(x_0,px_0)(t_n - \delta_{\Gamma_{\xi_p}})} \geq 1.
      \end{aligned}
\right.
$$

for all $\xi_p \in Par_{\Gamma}$, we define a $\Gamma$-invariant map on $\jD \cap \jH_{\xi_p}(u_{0,\xi_p})$ as follows:\\
for every $\widetilde{u}$ in $\jH_{\xi_p} (u_{0,\xi_p})$,

$$ \widetilde{H}_1(\widetilde{u})=
\left\{
      \begin{aligned}
	t_n + Y_n -\rho (\widetilde{u}) \text{on } \rho^{-1}(]Y_n,Y_n+t_n-t_{n+1}]) , \\
	t_{n+1} \text{on } \rho^{-1}(]Y_n+t_n-t_{n+1},Y_{n+1}]).
      \end{aligned}
\right.
$$
We extend $\widetilde{H}_1$ to $\widetilde{M}$ as follows.\\
First, we extend it to $\jD$ by a constant function such that 
$$ \widetilde{H}_1: \jD \to \R$$
is Hölder-continuous. Next, we extend it to $\widetilde{M}$ as follows. For all $\gamma \in \Gamma$, if $\widetilde{x} \in \gamma \jD$
then
$$ \widetilde{H}_1 (\widetilde{x}) = \widetilde{H}_1 (\gamma^{-1} \widetilde{x}). $$

Let $\widetilde{H}$ be the $\Gamma-$invariant potential obtained by pulling back ${H}_1$ on $T^1 \widetilde{M}$.
We denote by $H: T^1 M \to \R$ the map induced by $\widetilde{H}$ on $T^1 M$.

\begin{lem}(Coudène \cite{C04})
 $\widetilde{H}$ is a $\Gamma-$invariant Hölderian potential such that for all parabolic fixed point $\xi_p$, the critical exponent associated with
 $(\Gamma_{\xi_p},H)$ is of divergence type.
\end{lem}

We have therefore constructed a $\Gamma$-invariant Hölderian potential on $T^1 M$ such that the associated Gibbs
measure is finite and is supported on $\Omega$.\\

Remark that the divergence of the parabolic subgroups only depends on the value of the potential in the cusp. In the previous
construction, we made the assumption that the potential was constant on 
$$T^1 M_0 = T^1 (\jC \Lambda \Gamma \bigcup\limits_{\xi_p \in Par_{\Gamma}} \jH_{\xi_p}) / \Gamma .$$
However, taking
any bounded potential such that the resulting function on $T^1 M_0$ is Hölder-continuous, the associated Gibbs measure will still be
finite.\\
We now choose $p\in T^1 M = T^1\widetilde{M} / \Gamma$ such that $\phi_{t} (p)$ is periodic. We will denote by $\O(p)$ the closed curve
$\phi_{\R}(p)$ and assume that $\O(p) \subset T^1 M_0$.\\
For all $n \in \N$, we define a Lipschitz-continuous potential by

$$ F_n (x) = \max \{ c_n - c_n d (\O(p),x) ; H(x) \}.$$

\begin{prop}
 For all $n \in \N$, the critical exponent $\delta_{\Gamma,F_n}$ is finite. Moreover, we have 
 $$c_n \leq \delta_{\Gamma,F_n} \leq \delta_{\Gamma} + c_n.$$
\end{prop}
\begin{proof}
Since $c_k = \sup\limits_{x \in T^1 M} F_k(x)$, the upper bound $\delta_{\Gamma,F_k} \leq \delta_{\Gamma} + c_k$ is evident by the
second claim of Proposition \ref{prop:exp}. \\
Let $p$ be the periodic point of $T^1M$ defined above and $h\in \Gamma$ the generator of the isometry group fixing the periodic 
orbit $\phi_{\R}(p)$. Let $H= <h>$ and denote by $\ell$ the length of $h$. By the very definition of critical exponents, we have

\begin{center}\begin{align}
 \delta_{\Gamma,F_k} &= & \limsup\limits_{n \to \infty} \frac{1}{n} \log  \sum\limits_{\gamma \in \Gamma,\;n- \ell< d(x,\,\gamma x)\leq n} e^{\int_x^{\gamma x}\widetilde{F}_k} \\
  &\geq & \limsup\limits_{n \to \infty} \frac{1}{n} \log  \sum\limits_{\gamma \in H ,\;n-\ell< d(x,\,\gamma x)\leq n} e^{\int_x^{\gamma x}\widetilde{F}_k}.
\end{align}\end{center} 
Since the critical exponent does not depend on the choice of a base point, one can choose $x = \widetilde{p}$ where $\widetilde{p}$ is a lift of
$p$ on $T^1 \widetilde{M}$. Therefore, the fact that the
value of the potential on $\O(p)$ is constant, equal to $c_k$ implies that
$$ \int_{\widetilde{p}}^{\gamma {\widetilde{p}}}\widetilde{F}_k = d(x, \gamma x ) c_k \geq (n- \ell) c_k$$
which gives
\begin{center}\begin{align}
 \delta_{\Gamma,F_k} &\geq  \limsup\limits_{n \to \infty} \frac{1}{n} \log  \sum\limits_{\gamma \in H ,\;n-\ell< d(x,\,\gamma x)\leq n} e^{(n-\ell) c_k}
 \geq \limsup\limits_{n \to \infty} \frac{1}{n} \log \left( A_x e^{(n- \ell)c_k} \right),
\end{align}\end{center}
where $A_x = \sharp \{\gamma \in H : n- \ell < d(x,\gamma x) \leq n \}$.\\
Since the group $H$ is generated by a hyperbolic isometry $h$, for all $\gamma \in H$, there exists $i \in \N$ such that $\gamma = 
h^i$ and $\ell (\gamma)= \ell (h^i) = i \ell(h).$ Therefore the quantity $A_x$ does not depend on $n$ and 
$$ \delta_{\Gamma,F_k} \geq c_k.$$
which concludes the proof.
\end{proof}
Therefore, the Gibbs measure $m_{F_n}$ associated with $(\Gamma,F_n)$ of dimension $\delta_{\Gamma,F_n}$ exists for all $n\in \N$.\\

\section{Proof of the main theorem}
Let $D_{\O(p)}$ be the Dirac measure supported on $\O(p)$. We prove Theorem \ref{thm:princ} which states the following.\\
 \textit{
 Let $M$ be a geometrically finite, negatively curved manifold and $\phi_t$ its geodesic flow.
 If $\phi_t$ is topologically mixing on $\Omega$, then the set of probability measures that are mixing with respect
 to the geodesic flow is dense in $\M^1 (T^1 M)$ for the weak topology.}\\
 
 Here is our strategy: since the geodesic flow on a manifold with pinched negative curvature satisfies the 
 closing lemma (see \cite{MR1441541}) and  admits a local product structure, we use the following result.
 
 \begin{prop}(Coudène-Schapira \cite{MR2735038})
  Let $M$ be a complete, connected Riemannian manifold with pinched negative curvature and $\phi_t$ its geodesic flow.
  Then the set of normalized Dirac measures on periodic orbits is dense in the set of all invariant measures $\M^1 (T^1 M)$.
 \end{prop}
 It is therefore clear that the following proposition implies Theorem \ref{thm:princ}.
 \begin{prop} \label{prop:conv}
 For all $p \in T^1 M$ such that $\O(p)= \phi_{\R} (p)$ is periodic, there exists a sequence $\{m_k \}_{k \in \N}$ of measures satisfying
 the following properties
 \begin{enumerate}
  \item $m_k$ is a probability measure which is mixing with respect to the geodesic flow,
  \item $m_k \rightharpoonup D_{\O(p)}$.
 \end{enumerate}
 \end{prop}

Let $\htop(\phi_t)$ be the topological entropy of the geodesic flow on $T^1 M$ and $h_{\mu}(\phi_t)$ the measure theoretic entropy 
of the geodesic flow with respect to $\mu$. D. Sullivan proved in \cite{Su} that if the Bowen-Margulis measure $m_{BM}$
(\textit{i.e} the Gibbs measure associated with the potential $F=0$) is finite then
$$ \delta_{\Gamma} = h_{m_{BM}}(\phi_t) .$$

Using a result of C.J.Bishop and P.W.Jones \cite{BJ} connecting the critical exponent $\delta_{\Gamma}$ with the Hausdorff dimension of the
conical limit set of $\Gamma$, J.P.Otal and M.Peigné proved that for all $\phi_t-$invariant probability measures $m
\in \M^1(T^1\widetilde{M})$ which are not the Bowen-Margulis measure, we have the strict inequality,
$$ h_{\mu}(\phi_t) < \delta_{\Gamma} .$$
We refer to F. Ledrappier \cite{ENSML} for a survey of these results.

\begin{thm} \cite{otal2004}
Let $\widetilde{M}$ be a simply connected, complete Riemannian manifold with pinched negative curvature and $\Gamma$ be
a non-elementary discrete group of isometries of $\widetilde{M}$, then
$$ \htop (\phi_t)= \delta_{\Gamma}. $$
Moreover, there exists a probability measure $\mu$ maximizing the entropy if and only if the Bowen-Margulis measure is finite and
$\mu= m_{BM}$.
\end{thm}
We begin the proof of the main theorem. First, we state a general result which holds true for any metric space $X$ satisfying a variational
principle.
In the next claim, we suppose that $F: X \to \R$ is a measurable function such that there exist
an invariant compact set $K \subset X$ and a neighborhood $V$ of
$K$ satisfying the following assumptions:
\begin{center}
 $ \left\{
 \begin{aligned}
   \forall x \in K, F(x) = c = \sup\limits_{x \in X} F(x)\\
   \sup\limits_{x\in X \backslash V} F (x)= c' < c.
  \end{aligned}
  \right.$
 \end{center}
We say that a probability measure $\mu$ with finite entropy is an equilibrium state for a potential $F:X \to \R$ if it achieves the supremum of
 $$ m \mapsto h_m (\phi_t) + \int_X F dm $$
over all invariant probability measures with finite entropy.

\begin{lem}\label{lem:majvar}
 Let $X$ be a metric space, $\phi_t$ a flow defined on $X$ and $F:X \to \R$, $K$, $V$ defined as above.\\
 Suppose there exists an equilibrium state $\mu$ for $F$, then
 $$ \mu(X \backslash V) \leq \frac{h_{\mu}(\phi_t)}{c-c'}. $$
\end{lem}

\begin{proof}
 Let $m_K$ be a probability measure supported on $K$. Since $\mu$ realises the supremum of
 $$ m \in \M^1(X) \mapsto h_m (\phi_t) + \int_X F dm $$
 we have
 $$ h_{\mu} (\phi_t) + \int_X F d\mu \geq h_{m_K} (\phi_t) + \int_X F dm_K $$
 which implies
 \begin{equation}\label{eqn:1}
  h_{\mu} (\phi_t) + \int_X F d\mu \geq \int_X F dm_K
 \end{equation}
 since $h_{m_K} (\phi_t) \geq 0$. Moreover, since the potential $F$ is constant on $K$ we have
 $$ \int_X F dm_K = c $$
 which can be written as
 \begin{equation}\label{eqn:2}
  \int_X F dm_K = c (\mu(V) + \mu(X \backslash V)).
 \end{equation}

 Combining equations \ref{eqn:1} and \ref{eqn:2}, we obtain
 $$ 
 \begin{array}{rcl}
  h_{\mu} (\phi_t) &\geq& \int_X F dm_K - \int_X F d\mu \\
                   &\geq & c (\mu(V) + \mu(X \backslash V)) - c\mu(V) - c' \mu(X \backslash V) )
 \end{array}
 $$ 
 which finally gives
 $$ \frac{h_{\mu}(\phi_t)}{c-c'} \geq \mu(X \backslash V). $$
\end{proof}

Recall that a sequence $\{\mu_n\}_{n \in \N}$ of probability measures on a Polish space $X$ is tight if for all 
$\epsilon >0$ there exists a compact set $K \subset X$ such that 
$$ \forall n \in \N, \text{  } \mu_n(X \backslash K)< \epsilon .$$
We give a criterion for the convergence of a sequence of probability measures to the Dirac measure supported on the periodic orbit 
$\O(p)$. We denote by $V_{\epsilon}$ the subset of $T^1 M$ defined by

$$ V_{\epsilon}= \{ x \in T^1M : d(x, \O(p)) \leq \epsilon \}.$$

\begin{lem} \label{lem:cv}
 The following assertions are equivalent:
 \begin{enumerate}
  \item The sequence of probability measures $\{m^{F_n} \}_{n \in \N}$ converges to the Dirac measure supported on $\O(p)$, 
  \item for all $\epsilon >0$,
 $$ \lim\limits_{n\to \infty} m^{F_n}(T^1 M \backslash V_{\epsilon}) =0 . $$
 \end{enumerate}
\end{lem}
\begin{proof}It is clear that $(1) \Rightarrow (2)$. Let us show that $(2) \Rightarrow (1)$.\\
  We first notice that $\{m^{F_n} \}_{n \in \N}$ is tight. Indeed,
  let $V$ be a compact subset of $T^1 M$ containing $\O(p)$. Since condition $(2)$ is satisfied, for all $\epsilon >0$, there exists $N_0 >0$ such that
  for all $n \geq N_0,$
  $$ m^{F_n} (T^1 M \backslash V ) < \epsilon. $$
  For all $n \in \{1,..,N_0-1 \}$, we can also find a compact set $K_n$ such that
  $$ m^{F_n} (T^1 M \backslash K_n) < \epsilon.$$
  Define the compact set $K = (\bigcup\limits_{n = 1}^{N_0 -1} K_n) \cup V$.\\
  For all $n \in \N$, $K$ satisfies
  $$ m^{F_n}(T^1 M \backslash K) \leq \epsilon. $$
  Therefore, the sequence $\{m^{F_n} \}_{n \in \N}$ is tight.\\
  Since condition $(2)$ is satisfied and using the fact that the unique invariant probability measure supported on $\O(p)$ is $D_{\O(p)}$, each converging subsequence
  of $\{m^{F_n}\}$ converges to $D_{\O(p)}$.\\
  Therefore, by Prokhorov's Theorem \cite{Pro}, each sub sequence of $\{m^{F_n}\}$ possesses a further subsequence converging weakly to $D_{\O(p)}$ so the sequence
  $m^{F_n}$ converges weakly to $D_{\O(p)}. $
\end{proof}

We are now able to prove our main result of convergence. 
\begin{thm}
 The sequence $\{m^{F_n}\}_{n \in \N}$ of probability measures converges to the Dirac measure supported on $\O(p)$.
\end{thm}
\begin{proof}
 Let $D_{\O(p)}$ be the Dirac measure supported on the periodic orbit $\O(p)$. Recall that $c_n = \sup\limits_{x \in T^1 M} F_n(x)$.\\
 By lemma \ref{lem:cv}, we
 have to show that 
 $$\lim\limits_{n\to \infty} m^{F_n}(T^1 M \backslash V_{\epsilon}) =0 .$$
 Using the variational principle described in Theorem \ref{thm:varp}, we have
 $$ \delta_{\Gamma,F_n} = \sup\limits_{\mu \in \M^1(T^1M)} P_{\Gamma,F_n}(\mu) =  P_{\Gamma,F_n}(m^{F_n})$$
 and 
 $$ \delta_{\Gamma,F_n} \geq P_{\Gamma,F_n}(D_{\O(p)}) = h_{D_{\O(p)}}(\phi_t) +\int_{T^1 M} F_n dD_{\O(p)} .$$
 Since $\O(p)$ is invariant by the action of the geodesic flow and 
 $$h_{m_{F_n}} (\phi_t) \leq \delta_{\Gamma} < \infty,$$
 one can use the claim of lemma \ref{lem:majvar} to obtain the following inequality
 $$ m^{F_n}(T^1 M \backslash V_{\epsilon}) \leq \frac{\delta_{\Gamma}}{c_n -c'_n}$$
 where 
 $$c'_n = \sup\limits_{x \in T^1M \backslash V_{\epsilon}} F_n(x) .$$
 By the definition of the potential $F_n$, we know that
 $$c'_n \leq \max\{c_0,c_n(1-\epsilon) \}.$$
  Therefore, for all $\epsilon >0$ and $n$ large enough,
 $$m^{F_n}(T^1 M \backslash V_{\epsilon}) \leq \frac{\delta_{\Gamma}}{c_n \epsilon}$$
 implies that 
 $$ \lim\limits_{n\to \infty} m^{F_n}(T^1 M \backslash V_{\epsilon}) =0 .$$
 Which concludes the proof. 
\end{proof}

 Finally, we are able to prove corollary \ref{cor:gef} which states that the set of weak-mixing measures is a dense $G_{\delta}$ subset
 of $\M^1(T^1 M)$.\\
 Our proof relies on the following theorem.
 \begin{thm}(Coudène-Schapira \cite{MR3322793})
  Let $(\varphi^t)_{t\in \R}$ be a continuous flow on a Polish space. The set of weak-mixing measures on $X$ is a $G_{\delta}$ subset 
  of the set of Borel invariant probability measures on $X$.
 \end{thm}
 
 \begin{proof}(of corollary \ref{cor:gef})\\
 Since mixing measures are weak-mixing, Theorem \ref{thm:princ} implies that weak-mixing measures are dense in the set of probability 
 measures on $\Omega$. The previous theorem insures us that it is a $G_{\delta}$ set.
 
 Since negatively curved manifolds with pinched curvature satisfy the closing lemma, it is shown in \cite{MR2735038} that the set of
 invariant measures with full support on $\Omega$ is a dense $G_{\delta}$ subset of the set of invariant probability measures on $\Omega$.
 
 The intersection of those two dense $G_{\delta} $ sets is a dense $G_{\delta} $ set by the Baire Category Theorem.   
 \end{proof}
 
\section{Non-geometrically finite manifolds with cusps}

We now prove corollaries \ref{cor:geinf} and \ref{cor:surf}.
First of all, remark that since the manifold possesses a cusp then from lemma \ref{lem:mix}, the geodesic flow is topologically mixing on $T^1M= T^1\widetilde{M} /
\Gamma$. \\
Let $\gamma \in \Gamma$ and $D(\gamma)$ the subset of $\widetilde{M}$ bounded by
$$\{ \widetilde{x} \in \widetilde{M} : d(\widetilde{x},\widetilde{x}_0) = d(\widetilde{x},\gamma \widetilde{x}_0) \}$$
and containing $\gamma \widetilde{x_0}$. We define the subset $C_{\gamma}$ on $\partial_{\infty} \widetilde{M}$ as follows.
$$ C_{\gamma} = \partial_{\infty} \widetilde{M} \cap \overline{D(\gamma)}.$$

The proof of corollary \ref{cor:geinf} is deduced from the following result.
\begin{lem}\label{lem:V_0}
 Let $M$ be a connected, complete pinched negatively curved manifold with a cusp and $\phi_t$ the geodesic flow defined on $T^1 M$.
 There exists a geometrically finite manifold $\hat{M}$ such that its geodesic flow $\hat{\phi}_t$ is topologically mixing and a
 covering map $ \rho: T^1 \hat{M} \to T^1M $ such that the diagram
 $$
 \begin{CD}
  T^1 \hat{M} @>{\hat{\phi}_t}>> T^1 \hat{M} \\
  @V{\rho}VV  @VV{\rho}V \\
  T^1 M @>{\phi_t}>> T^1 M
 \end{CD}
 $$
 commutes.
\end{lem}

\begin{proof}
Let $\xi_p \in \partial_{\infty}\widetilde{M}$ be a bounded parabolic point fixed by a parabolic isometry $\gamma_p \in \Gamma$ and 
$h \in \Gamma$ be a hyperbolic transformation.\\
Let $N>0$ be defined such that the sets $C_{\gamma_p^N} , C_{h^N}, C_{\gamma_p^{-N}} , C_{h^{-N}}$ have 
disjoint interiors. We define $\Gamma_0 =$ $<\gamma_p^N , h^N>$, a subgroup of $\Gamma$. The ping-pong lemma shows that $\Gamma_0$ is a
discrete group which acts freely discontinuously on $\widetilde{M}$. So, the quotient $\widetilde{M} / \Gamma_0$ is a geometrically
finite manifold.
\end{proof}

\begin{proof}(of corollary \ref{cor:geinf})
Let $\hat{M}$ be a geometrically finite manifold constructed as in lemma \ref{lem:V_0} and $\rho$ its associated covering map. We
can use the proof of Theorem \ref{thm:princ} on $\hat{M}$ and construct a sequence $\hat{m}_k$ of invariant mixing measures for 
$\hat{\phi}_t$ such that $\hat{m}_k \rightharpoonup D_{\O(p)}$.\\
Since the geodesic flow $\phi_t$ is a factor of $\hat{\phi}_t$, we can define $\nu_k$ to be the push-forward by $\rho$
of $\hat{m}_k$, then it is an invariant mixing measure on $T^1M$ and for all bounded continuous function $g$ on $T^1 M$,
$$ 
\begin{array}{rcl}
 \lim\limits_{k \to \infty} \int_{T^1M} g d\nu_k &=& \lim\limits_{k \to \infty} \int_{T^1\hat{M}} g\circ \rho dm_k \\
  &=& \int_{T^1 \hat{M}} g\circ \rho dD_{\O(p)} \\
  &=& \int_{T^1 M} g d(\rho_* D_{\O(p)}).
\end{array}
$$
\end{proof}

We end up by the proof of corollary \ref{cor:surf}. In the case of a surface $S$ (or a constant negatively curved manifold), we don't 
need to ask for the existence of a bounded parabolic point. Since the geodesic flow is always topologically mixing in restriction to its
non wandering set, we can choose two hyperbolic isometries $h_1, h_2$ in $\Gamma$ such that the subgroup
$\Gamma_0 = <h_1, h_2>$ is convex-cocompact.\\
The same proof as lemma \ref{lem:V_0} shows us that the geodesic flow $\phi_t$ on $S$ is a factor of the geodesic flow $\hat {\phi_t}$ 
on the convex-cocompact manifold $T^1 S_0 = T^1\widetilde{S} / \Gamma_0 $. Therefore, the previous proof gives us the density result.

\end{document}